\documentclass{birkjour}
\usepackage{amssymb,amsthm,amsmath,amsfonts,mathrsfs}
\usepackage{xcolor}
\newtheorem{theorem}{Theorem}[]
\newtheorem{definition}[theorem]{Definition}
 \newtheorem*{example}{Example}

\newtheorem{corollary}[theorem]{Corollary}

\begin{document}

%-------------------------------------------------------------------------
% editorial commands: to be inserted by the editorial office
%
%\firstpage{1} \volume{228} \Copyrightyear{2004} \DOI{003-0001}
%
%
%\seriesextra{\xiiust an add-on}
%\seriesextraline{This is the Concrete Title of this Book\br H.E. R and S.T.C. W, Eds.}
%\firstpage{1}
%\revised{\xi une 1, 2000}
%\accepted{\xi uly 22, 2000}

\title[Solvability of Vekua-type periodic operators]
{Solvability of Vekua-type periodic operators and applications to classical equations}

%----------Author 1
\author[Kirilov]{Alexandre Kirilov}
\email{akirilov@ufpr.br}
\address{
	Departamento de Matem\'atica \br 
	Universidade Federal do Paran\'a \br 
	CP 19096, CEP 81531-990, Curitiba\br 
	Brasil}

\thanks{This study was financed in part by the Coordenação de Aperfeiçoamento de Pessoal de Nível Superior - Brasil (CAPES) - Finance Code 001. The first and second authors were supported in part by CNPq - Brasil (grants 316850/2021-7 and 423458/2021-3).}

%----------Author 2
\author[Moraes]{Wagner Augusto Almeida de Moraes}
\address{
	Departamento de Matem\'atica \br 
	Universidade Federal do Paran\'a \br 
	Caixa Postal 19096\\  CEP 81531-990, Curitiba, Paran\'a \br 
	Brasil}
\email{wagnermoraes@ufpr.br}

%----------Author 3
\author[Tokoro]{Pedro Meyer Tokoro}
\address{
	Programa de P\'os-Gradua\c c\~ao em Matem\'atica \br 
	Universidade Federal do Paran\'a \br 
	Caixa Postal 19096\\  CEP 81531-990, Curitiba, Paran\'a \br 
	Brasil}
\email{tokoro.p@gmail.com}



\date{today}
%%% ----------------------------------------------------------------------

\begin{abstract}
In this note, we investigate Vekua-type periodic operators of the form $Pu=Lu-Au-B\bar u$, where $L$ is a constant coefficient partial differential operator. We provide a complete characterization of the necessary and sufficient conditions for the solvability and global hypoellipticity of $P$. As an application,  we provide a comprehensive characterization of Vekua-type operators associated with classical wave, heat, and Laplace equations.
\end{abstract}

%%% ----------------------------------------------------------------------
\maketitle
%%% ----------------------------------------------------------------------

	\section{Introduction}

In \cite{vekua}, the Georgian mathematician I. N. Vekua introduced the theory of generalized analytic functions, which are the solutions to equations of the form
\begin{equation*}\label{vekua_original}
	\partial_{\bar{z}}u + Au + B\bar{u} = F,
\end{equation*}
where $\partial_{\bar{z}} = (\partial_{x}+i\partial_{y})/2$ and the coefficients $A$ and $B$ belong to a suitable function space in the complex plane. This theory is closely related to the theory of holomorphic functions, assuming some regularity on the coefficients to apply the well-known similarity principle. Vekua also applied this theory to explore problems in the membrane theory of shells and their connections with the problem of infinitesimal bendings of surfaces. 

In \cite{kravchenko}, V. Kravchenko extended this theory by replacing the Cauchy-Riemann operator with more general complex-valued vector fields. This extension allowed the study of well-known equations from mathematical physics, such as the Schrödinger, Dirac, and Maxwell equations, among others. Furthermore, connections between second-order elliptic equations and Vekua-type equations were established.

The contributions of Vekua, Kravchenko, and several other authors have found wide applications, including studies in boundary value problems in elasticity theory, hydrodynamics, electric potential, mechanics, and more.

In this paper we investigate the solvability of the operator $P$ on the $n$-dimensional torus $\mathbb{T}^n \simeq \mathbb{R}^n / 2\pi\mathbb{Z}^n$, defined by
\begin{equation}\label{P}
	Pu=Lu-Au-B\bar u,
\end{equation}
where $A, B \in \mathbb{C}$, and $L:\mathcal{C}^\infty(\mathbb{T}^n)\to \mathcal{C}^\infty(\mathbb{T}^n)$ is a partial differential operator.

Our main contribution lies in establishing both necessary and sufficient conditions for the solvability of the operator $P$. Furthermore, we establish the equivalence between solvability and global hypoellipticity. In conclusion, we apply this result to study Vekua-type operators associated with classical operators, characterizing their solvability and global hypoellipticity.

Our approach is inspired by results in \cite{BDM} and \cite{AD}. Additional insights about solvability on the torus can be obtained in \cite{BCP,BDG2017_jfaa,BDGK2015_jpdo,hounie,Petr2011_tams} and related works. In the broader context of compact Lie groups, some initial results are presented in \cite{wagner,KMR2020_bsm,KMR2021_jfa}. Additional references regarding infinitesimal deformations of surfaces connected with solvability of Vekua-type operators include \cite{LeMe22_2,mez2,mez4}  and references therein.

\section{Necessary and sufficient conditions for solvability}

Consider the operator $P: \mathcal{C}^\infty(\mathbb{T}^n) \to \mathcal{C}^\infty(\mathbb{T}^n)$ defined by
\begin{equation}\label{P}
	Pu = Lu - Au - B\bar{u},
\end{equation}
where \(A, B \in \mathbb{C}\), and \(L\) is a differential operator of the form
$$Lu = \sum_{0<|\alpha|\leqslant m}c_\alpha\partial^\alpha,$$
where \(c_\alpha \in \mathbb{C}\), for all \(\alpha \in \mathbb{N}_0^n\) satisfying \(0 < |\alpha| \leqslant m\), with symbol 	
$$\sigma_L(\xi) = \sum_{0<|\alpha|\leqslant m}i^{|\alpha|}c_\alpha \xi^\alpha, \ \xi \in \mathbb{Z}^n.$$

By the continuity of \(P\), if \(u(x) = \sum_{\xi \in \mathbb{Z}^n} \widehat{u}(\xi)e^{i\xi\cdot x}\) is a distribution in \(\mathscr{D}'(\mathbb{T}^n)\) then
\begin{align*}
	Pu(x,t) = & \sum_{\xi\in\mathbb{Z}^n} (L-A)\left(\widehat{u}(\xi) e^{i\xi\cdot x}\right) - B\sum_{\xi\in\mathbb{Z}^n}\overline{\widehat{u}(\xi) e^{i\xi\cdot x}}\\
	=  & \sum_{\xi\in\mathbb{Z}^n} \left(\sigma_L(\xi)-A\right)\widehat{u}(\xi) e^{i\xi\cdot x} - B\sum_{\xi\in\mathbb{Z}^n}\overline{\widehat{u}(-\xi)}e^{-i\xi\cdot x}.
\end{align*}

Therefore, the Fourier coefficients of any solution of $Pu=f$ must satisfy the following equation:
$$(\sigma_L(\xi)-A)\widehat{u}(\xi) - B\overline{\widehat{u}(-\xi)} = \widehat{f}(\xi).$$

Taking the conjugate of the previous equation for $-\xi\in\mathbb{Z}^n$, we obtain
$$
\overline{\widehat{f}(-\xi)}  =  (\overline{\sigma_L(-\xi)}-\bar A)\overline{\widehat{u}(-\xi)}-\bar B\widehat{u}(\xi)
$$
which gives us, for each $\xi\in\mathbb{Z}^n$, the following linear system:
$$
\begin{cases}
	\ \ \left(\sigma_L(\xi)-A\right)\widehat{u}(\xi) \ - \ B\overline{\widehat{u}(-\xi)} \ \ = \  \widehat{f}(\xi)\\
	-\bar B\widehat{u}(\xi)+(\overline{\sigma_L(-\xi)}-\bar A) \overline{\widehat{u}(-\xi)} =  \overline{\widehat{f}(-\xi)}
\end{cases}
$$

Solving this system for $\widehat{u}(\xi)$, we obtain
\begin{equation}\label{Deltau} 
	\Delta_\xi \widehat{u}(\xi)=(\overline{\sigma_L(-\xi)}-\bar A) \widehat{f}(\xi)+B\overline{\widehat{f}(-\xi)},
\end{equation}
where the discriminant
\begin{equation}\label{Delta_xi} 
	\Delta_\xi=\left(\sigma_L(\xi)-A\right)\cdot ( \overline{\sigma_L(-\xi)}-\bar A)-|B|^2.
\end{equation}

Observe that $\overline{\Delta_\xi}=\Delta_{-\xi}$, for all $\xi\in\mathbb{Z}^n$. In particular, $\Delta_\xi=0$ if, and only if, $\Delta_{-\xi}=0$.

We consider the same notion of solvability used in \cite{BDM} and \cite{AD} and the usual notion of global hypoellipticity.

\begin{definition}
	We say that an operator $P:\mathcal{C}^\infty(\mathbb{T}^n)\to \mathcal{C}^\infty(\mathbb{T}^n)$ is solvable if exists a subspace $\mathcal{F}\subset \mathcal{C}^\infty(\mathbb{T}^n)$ of finite codimension such that, for all $f\in\mathcal{F}$, exists $u\in \mathcal{C}^\infty(\mathbb{T}^n)$ such that $Pu=f$. Also, $P$ is globally hypoelliptic if $u\in\mathcal{D}'(\mathbb{T}^n)$ and $Pu\in \mathcal{C}^\infty(\mathbb{T}^n)$ imply that $u\in \mathcal{C}^\infty(\mathbb{T}^n)$. 
\end{definition}

\begin{theorem}\label{DC}
	The operator \(P\) is solvable if and only if the following Diophantine condition holds: there exists \(\gamma > 0\) such that 
	\begin{equation}
	 	\|\xi\|\geqslant\gamma\ \Rightarrow\ |\Delta_\xi|\geqslant\|\xi\|^{-\gamma}. \tag{DC} \label{DC_condition}
	\end{equation}
\end{theorem}

\begin{proof}
	
This demonstration follows the same procedure as the proof of Theorem $1$ given in \cite{AD}. Let us start by assuming that the condition \eqref{DC_condition} holds. In this context, the set $\Lambda = \{\xi \in \mathbb{Z}^n : \|\xi\| < \gamma\}$ is finite, and the discriminant $\Delta_\xi$ is non-zero for all $\xi \in \mathbb{Z}^n \setminus \Lambda$. 	Additionally, we have $\Delta_{\xi}^{-1} \leqslant \|\xi\|^\gamma,$ for every $\xi \in \mathbb{Z}^n \setminus \Lambda$.

If the Fourier coefficients of $f \in \mathcal{C}^\infty(\mathbb{T}^n)$ satisfy the compatibility condition 
\begin{equation}\label{comp_cond} 
	(\overline{\sigma_L(-\xi)}-\bar A)\widehat{f}(\xi) + B\overline{\widehat{f}(-\xi)} = 0, \ \xi \in \Lambda
\end{equation} 
By (\ref{Deltau}), we set
\begin{equation*} 
	\widehat{u}(\xi) = \dfrac{(\overline{\sigma_L(-\xi)}-\bar A)\widehat{f}(\xi) + B\overline{\widehat{f}(-\xi)}}{\Delta_\xi},
\end{equation*} 
and thus $u = \sum_{\xi \in \mathbb{Z}^n \setminus \Lambda} \widehat{u}(\xi) e^{i\xi\cdot x}$ is a solution of $Pu=f$. 

Note that the sequence $\{\widehat{f}(\xi)\}$ decays rapidly, implying the rapid decay of $\{\widehat{u}(\xi)\}$, and consequently, $u \in \mathcal{C}^\infty(\mathbb{T}^n)$.

Finally, given the finiteness of $\Lambda$, the number of compatibility conditions on $f \in \mathcal{C}^\infty(\mathbb{T}^n)$ ensuring the existence of a solution $u \in \mathcal{C}^\infty(\mathbb{T}^n)$ to $Pu=f$ is finite. Therefore, according to our definition, $P$ is solvable.

\medskip	
On the other hand, let us assume that \eqref{DC_condition} does not hold. Thus, for each $\ell\in\mathbb{N}$, there exists a $\xi_\ell\in\mathbb{Z}^n$ such that 
$$\|\xi_\ell\|\geqslant \ell \mbox{ and } |\Delta_{\xi_\ell}|<\|\xi_\ell\|^{-\ell}.$$

It is worth noting that $\|\xi_\ell\|\to\infty$, since $\|\xi_\ell\|\geqslant \ell$ for all $\ell\in\mathbb{N}$. Denoting $\xi_\ell=(\xi_{1\ell},\dots,\xi_{n\ell})$, and by passing to a subsequence if necessary, we may assume that there is a coordinate $j\in\{1,\dots,n\}$ such that $\xi_{j\ell}$ is non-negative and maintains the same sign, for all $\ell\in\mathbb{N}$.

Now, consider the set $\Omega=\{\xi_\ell\in\mathbb{Z}^n: \ell\in\mathbb{N}\}$. The previous choice of the sequence $\{\xi_\ell\}_{\ell\in\mathbb{N}}$ implies that, if $\xi\in\Omega$, then $-\xi\notin\Omega$. Particularly, for any $\Omega_0\subset\Omega$, if $\xi\in\Omega_0$, then $-\xi\notin\Omega_0$.
	
\medskip	
\paragraph{Case 1:} $\Delta_{\xi_\ell}=0$ for infinitely many indices $\ell\in\mathbb{N}$.

Upon considering a subsequence, we can assume that $\Delta_{\xi_\ell}=0$ for all $\ell\in\mathbb{N}$. Consequently, by (\ref{Deltau}), we have
\begin{equation}\label{comp_cond2}
	(\overline{\sigma_L(-\xi)}-\bar A) \widehat{f}(\xi)+B\overline{\widehat{f}(-\xi)}=0, \  \xi\in\Omega.
\end{equation}

In the scenario where $B\neq 0$, utilizing \eqref{Delta_xi}, we obtain
\begin{equation*} 
	(\sigma_L(\xi_\ell)-A)\cdot (\overline{\sigma_L(-\xi_\ell)}-\bar A)=|B|^2, \ \ell\in\mathbb{N}.
\end{equation*}

As $|B|^2\neq 0$, we can conclude that $\sigma_L(\xi_\ell)-A\neq 0$ and $\overline{\sigma_L(-\xi_\ell)} - \bar A\neq 0$ for all $\ell\in\mathbb{N}$. In this context, (\ref{comp_cond2}) implies the existence of infinitely many compatibility conditions for the Fourier coefficients of $f\in \mathcal{C}^\infty(\mathbb{T}^n)$ to satisfy $Pu=f$. Consequently, $P$ is not solvable according to our definition.

Now, let us consider the case where $B=0$. In this scenario, the Fourier coefficients of $u$ and $f$ must satisfy, for all $\xi\in\mathbb{Z}^n$,
\begin{equation*} 
	(\sigma_L(\xi)-A)\widehat{u}(\xi)=\widehat{f}(\xi)\quad\text{and}\quad (\overline{\sigma_L(-\xi)}-\bar A)\overline{\widehat{u}(-\xi)}=\overline{\widehat{f}(-\xi)}
\end{equation*} 

Since
\begin{equation*} 
	0=\Delta_{\xi_\ell}=[\sigma_L(\xi_\ell)-A]\cdot[\overline{\sigma_L(-\xi_\ell)}-\bar A], \ \ell\in\mathbb{N}
\end{equation*} 
it implies that, for each $\ell\in\mathbb{N}$, at least one of the following is satisfied:  $$\sigma_L(\xi_\ell)-A=0\quad\text{or}\quad\overline{\sigma_L(-\xi_\ell)}-\bar A=0,$$
which further implies that $\widehat{f}(\xi)=0$ or $\widehat{f}(-\xi)=0$, for infinitely many $\xi\in\mathbb{Z}^n$. 

Therefore, in this case, there exist infinitely many compatibility conditions for the Fourier coefficients of $f$ such that $Pu=f$ has a smooth solution. Consequently, $P$ is not solvable.

\medskip
\paragraph{Case 2:} $\Delta_{\xi_\ell}=0$ for a finite number of indices $\ell \in\mathbb{N}$.

By passing to a subsequence, we may assume $\Delta_{\xi_\ell}\neq 0$ for all $\ell\in\mathbb{N}$, i.e.,
\begin{equation}\label{Delta_xi_ell} 
	0<|\Delta_{\xi_\ell}|<\|\xi_\ell\|^{-\ell}, \ \ell\in\mathbb{N}.
\end{equation}

Assume $B\neq 0$ and consider an infinite subset $\Omega_0\subset\Omega$. In this case, 
\begin{equation*} 
	f(x)=\sum_{\xi\in\Omega_0}\Delta_\xi e^{i\xi\cdot x}
\end{equation*} 
defines a smooth function $f\in \mathcal{C}^\infty(\mathbb{T}^n)$, due to \eqref{Delta_xi_ell}.

If $u\in\mathscr{D}'(\mathbb{T}^n)$ is a solution of $Pu=f$, the projection of $u$ on the subspace of $\mathscr{D}'(\mathbb{T}^n)$ generated by the frequencies $\pm\Omega_0$ is
\begin{equation*} 
	v(x)=\sum_{\xi\in\Omega_0}(\overline{\sigma_L(-\xi)}-\bar{A})e^{i\xi\cdot x}+ \sum_{\xi\in\Omega_0} Be^{-i\xi\cdot x}.
\end{equation*}

In fact, if $\xi\in\Omega_0$, then $\widehat{f}(\xi)=\Delta_\xi$ and $\overline{\widehat{f}(-\xi)}=0$. It follows from (\ref{Deltau}) that
\begin{equation*} 
	\Delta_\xi \widehat{u}(\xi)=[\overline{\sigma_L(-\xi)}-\bar A]\widehat{f}(\xi) = [\overline{\sigma_L(-\xi)}-\bar A]\Delta_\xi\ \Rightarrow\ \widehat{u}(\xi) = \overline{\sigma_L(-\xi)}-\bar A.
\end{equation*} 

If $-\xi\in\Omega_0$, then $\widehat{f}(\xi)=0$ and $\overline{\widehat{f}(-\xi)}=\overline{\Delta_{-\xi}}=\Delta_\xi$. It follows from (\ref{Deltau}) that
\begin{equation*} 
	\Delta_{\xi}\widehat{u}(\xi)=B\Delta_{\xi} \ \Rightarrow \ \widehat{u}(\xi)=B. 
\end{equation*} 

Observe that $v\in\mathscr{D}'(\mathbb{T}^n)\setminus \mathcal{C}^\infty(\mathbb{T}^n)$, implying $u\notin \mathcal{C}^\infty(\mathbb{T}^n)$. Since this construction is valid for any infinite subset $\Omega_0\subset\Omega$, we obtain infinitely many linearly independent functions $f\in \mathcal{C}^\infty(\mathbb{T}^n)$ for which there are no corresponding $u\in \mathcal{C}^\infty(\mathbb{T}^n)$  solutions to $Pu=f$. Therefore, $P$ is not solvable.

\medskip
Now we assume $B=0$ and consider an infinite subset $\Omega_0\subset\Omega$. It follows from (\ref{Delta_xi_ell}) that
\begin{equation*} 
	|\Delta_{\xi_\ell}|=|\sigma_L(\xi_\ell)-A|\cdot|\sigma_L(-\xi_\ell)-A|\leqslant \|\xi_\ell\|^{-\ell}, \ \ell\in \mathbb{N}. 
\end{equation*} 

This implies that, 
$$|\sigma_L(\xi_\ell)-A|\leqslant \|\xi_\ell\|^{-\ell/2} \mbox{ or } |\sigma_L(-\xi_\ell)- A|\leqslant \|\xi_\ell\|^{-\ell/2}, \ \ell\in\mathbb{N}.$$ 

Let us suppose that $|\sigma_L(\xi_\ell)-A|\leqslant \|\xi_\ell\|^{-\ell/2}$ for infinitely many $\ell\in\mathbb{N}$. Passing to a subsequence, if necessary, we may assume that
\begin{equation}\label{decay_Beq0}
	|\sigma_L(\xi_\ell)-A|\leqslant \|\xi_\ell\|^{-\ell/2}, \ \ell\in\mathbb{N}.
\end{equation}

It follows from (\ref{decay_Beq0}) that
\begin{equation*}
	f(x)=\sum_{\xi\in\Omega_0}(\sigma_L(\xi)-A)e^{i\xi\cdot x}
\end{equation*}
defines a function $f\in \mathcal{C}^\infty(\mathbb{T}^n)$. 

If $u\in\mathscr{D}'(\mathbb{T}^n)$ is a solution of $Pu=f$, the projection of $u$ on the subspace of $\mathscr{D}'(\mathbb{T}^n)$ generated by the frequencies $\Omega_0$ is
\begin{equation*} 
	v(x)=\sum_{\xi\in\Omega_0}e^{i\xi\cdot x}.
\end{equation*} 

In fact, if $\xi\in\Omega_0$, (\ref{Deltau}) implies	
\begin{equation*} 
	\underbrace{[\sigma_L(\xi)-A]\cdot[\overline{\sigma_L(-\xi)}-\bar A]}_{=\Delta_\xi}\widehat{u}(\xi) =[\overline{\sigma_L(-\xi)}-\bar A]\cdot \underbrace{[\sigma_L(\xi)-A]}_{=\widehat{f}(\xi)}.
\end{equation*} 

Hence, $\widehat{u}(\xi)=1$, for any $\xi\in\Omega_0$. Thus $v\in\mathscr{D}'(\mathbb{T}^n)\setminus \mathcal{C}^\infty(\mathbb{T}^n)$, which implies $u\notin \mathcal{C}^\infty(\mathbb{T}^n)$.

If we have $|\sigma_L(-\xi_\ell)- A|\leqslant \|\xi_\ell\|^{-\ell/2}$ for all $\ell\in\mathbb{N}$ (passing to a subsequence), take $f$ as in the previous case with $\xi\in-\Omega_0$.

Again, observe that this construction is valid for all infinite subsets $\Omega_0\subset\Omega$. Then, we obtain infinitely many linearly independent functions $f\in \mathcal{C}^\infty(\mathbb{T}^n)$ such that there is no $u\in \mathcal{C}^\infty(\mathbb{T}^n)$ solution of $Pu=f$. Therefore, $P$ is not solvable.
\end{proof}

\begin{corollary}\label{GH}
	$P$ is solvable if, and only if, $P$ is globally hypoelliptic.
\end{corollary}

\begin{proof}
Let $u \in \mathscr{D}'(\mathbb{T}^n)$ be a solution to $Pu = f \in \mathcal{C}^\infty (\mathbb{T}^n)$. Since $P$ is solvable, the condition \eqref{DC_condition} holds. Therefore, there exists $\gamma > 0$ such that
\begin{equation*}
	\|\xi\| \geqslant \gamma \implies \widehat{u}(\xi) = \dfrac{(\overline{\sigma_L(-\xi)}-\bar A)\widehat{f}(\xi) + B\overline{\widehat{f}(-\xi)}}{\Delta_\xi}.
\end{equation*} 

Consequently, the rapid decay of the sequence $(\widehat{u}(\xi))$ is a consequence of the rapid decay of $(\widehat{f}(\xi))$. This implies that $u \in \mathcal{C}^\infty (\mathbb{T}^n)$ and $P$ is globally hypoelliptic.
	
On the other hand, assume that $P$ is not solvable. It follows from Theorem \ref{DC} that there exists $f\in \mathcal{C}^\infty(\mathbb{T}^n)$ and $u\in\mathscr{D}'(\mathbb{T}^n)\setminus \mathcal{C}^\infty(\mathbb{T}^n)$ such that $Pu=f$. Therefore, $P$ is not globally hypoelliptic.
\end{proof}

\begin{theorem}\label{elliptic}
	If $L$ is an elliptic differential operator, then  $P$ is solvable and globally hypoelliptic.
\end{theorem}
\begin{proof}
Given that $L-A$ is elliptic for any $A\in \mathbb{C}$, we can assume, without loss of generality, that $P$ is of the form $Pu=Lu-B\bar u$, with $B\in\mathbb{C}$.

Let $m$ be the order of $L$. Due to the ellipticity of $L$, there are positive constants $R_0$ and $M$ such that
\begin{equation*} 
	\|\xi\|\geqslant R_0\ \Rightarrow\ |\sigma_{L}(\xi)|\geqslant M\|\xi\|^m.
\end{equation*} 

For $\|\xi\|\geqslant R_0$, we have
\begin{align*}
	|\Delta_{\xi}| \ & = |\sigma_{L}(\xi)\cdot\overline{\sigma_{L}(-\xi)}-|B|^2|  \geqslant M^2\|\xi\|^{2m}-|B|^2.
\end{align*}

Choosing $R\geqslant R_0$ sufficiently large so that $M^2R^2>|B|^2+1$, we ensure that $\|\xi\|\geqslant R$ implies $|\Delta_{\xi}|>1$. 

Note that the set $\Lambda=\{\xi\in\mathbb{Z}^{n}: \|\xi\|<R\}$ is finite, and, in particular, $\Lambda_0=\{\xi\in\mathbb{Z}^n: \Delta_\xi=0\}\subset\Lambda$ is also finite. 

Let $u\in\mathscr{D}'(\mathbb{T}^n)$ be a solution to $Pu=f$, and suppose $\widehat{f}(\xi)=0$ for all $\xi\in\Lambda_0$. In this case, we have
\begin{equation*} 
	\widehat{u}(\xi)=\dfrac{\overline{\sigma_L(-\xi)}\widehat{f}(\xi)+B\overline{\widehat{f}(-\xi)}}{\Delta_\xi}, \ \xi\in\Lambda_0.
\end{equation*} 

Therefore, the sequence $(\widehat{u}(\xi))$ decays rapidly, resulting in $u\in \mathcal{C}^\infty(\mathbb{T}^n)$. It is crucial to note that we have only finitely many compatibility conditions over $f\in \mathcal{C}^\infty(\mathbb{T}^n)$ for the equation $Pu=f$ to admit a solution $u\in \mathcal{C}^\infty(\mathbb{T}^n)$. This finiteness arises from the fact that $\Lambda_0$ is finite. Consequently, $P$ is solvable. The global hypoellipticity is a direct consequence of Corollary \ref{GH}.
\end{proof}

\begin{example}[Laplace operator]
If $L = \sum_{j=1}^{n} {\partial^2}/{\partial x_j^2}$ then the operator $P$ given by $Pu=Lu-Au-B\bar u$, with $A,B\in\mathbb{C}$ is solvable and globally hypoelliptic.
\end{example}

\section{Applications}

In this section, we present results regarding the solvability and global hypoellipticity of some classical examples of constant-coefficients operators.

\begin{theorem}[Heat Operator]\label{heat}
	If $L={\partial}/{\partial t}-\eta^2\sum_{j=1}^{n}{\partial^2}/{\partial x_j^2}$, where $\eta>0$, then $P$ is solvable and globally hypoelliptic.
\end{theorem}

\begin{proof}
To prove that the discriminant associated with $P$ satisfies the condition \eqref{DC_condition} of Theorem \ref{DC}, we begin by noting that 
$$\sigma_L(\tau,\xi) = i\tau + \eta^2\|\xi\|^2 = \overline{\sigma_L(-\tau,-\xi)},$$ 
for all $(\tau,\xi)\in \mathbb{Z}^{n+1}.$ By \eqref{Delta_xi} we obtain
\begin{equation*}
		\Delta_{\tau,\xi} 
		= \eta^4\|\xi\|^4 - 2\text{Re}(A) \eta^2\|\xi\|^2 - \tau^2 + |A|^2 - |B|^2 + 2i\tau (\eta^2\|\xi\|^2-\text{Re}(A)).
\end{equation*}

For $\tau=0$, there exists $\gamma_1 > 0$ such that
\begin{align}
	\|\xi\| \geqslant \gamma_1 \implies  |\Delta_{0,\xi}| \geqslant & |\text{Re}(\Delta_{0,\xi})| \label{tau=0} \\
	= & \Big| \eta^4|\xi|^4 - 2\text{Re}(A)\eta^2|\xi|^2 + |A|^2 - |B|^2 \Big| > 1. \nonumber
\end{align}

On the other hand, when $\tau\neq 0$, we find
\begin{equation*} 
	|\Delta_{\tau,\xi}| \geqslant |\text{Im}(\Delta_{\tau,\xi})| =  2|\tau| \big| \eta^2\|\xi\|^2-\text{Re}(A) \big|.
\end{equation*} 
	
Note that the condition $\eta^2|\xi|^2 - \text{Re}(A) = 0$ holds for at most a finite number of indices $(\tau,\xi) \in \mathbb{Z}^{n+1}$. Consequently, there exists $\gamma_2 > 0$ such that	
$ \big|\eta^2\|\xi\|^2 - \text{Re}(A)\big| > 1/2,$ whenever $\|\xi\| \geqslant \gamma_2.$ 
Thus, 
\begin{equation}
	\|\xi\| \geqslant \gamma_2 \implies |\Delta_{\tau,\xi}| > 1.  \label{tau_neq_0}
\end{equation}

Now, let us consider the remaining case where $(\tau,\xi) \in \mathbb{Z}^{n+1}$ satisfies $\eta^2\|\xi\|^2 = \text{Re}(A)$. In this scenario, we have
\begin{equation*}
	|\Delta_{\tau,\xi}|\ \geqslant\ |\text{Re}(\Delta_{\tau,\xi})| 	= \big|\tau^2+\text{Re}(A)^2+|B|^2-|A|^2\big|, 
\end{equation*}
indicating that $|\text{Re}(\Delta_{\tau,\xi})| = 0$ if and only if
\begin{equation*}
	\tau = \pm\sqrt{-\text{Re}(A)^2 + |A|^2 - |B|^2}.
\end{equation*} 

Therefore, $|\Delta_{\tau,\xi}|$ vanishes for at most two integers $\tau \in \mathbb{Z}$.  So, let $\gamma_3 > 0$ be such that
\begin{equation}
	|\tau| \geqslant \gamma_3 \implies  |\Delta_{\tau,\xi}| \geqslant \big|\tau^2 + \text{Re}(A)^2 + |B|^2 - |A|^2\big| > 1. \label{tau_remaining}
\end{equation}

Define $\gamma \doteq 2\max\{\gamma_1,\gamma_2,\gamma_3\}$. Consequently, when $\|\xi\|+|\tau| \geqslant \gamma$, we can ensure that either $\|\xi\| \geqslant \gamma/2$ or $|\tau| \geqslant \gamma/2$.

For $\|\xi\| \geqslant \gamma/2$, if $\tau=0$, then $|\Delta_{0,\xi}| \geq |\text{Re}(\Delta_{0,\xi})|>1$ since $\gamma \geq \gamma_1$. Similarly, for $\tau \neq 0$, we have $|\Delta_{\tau,\xi}| \geq |\text{Im}(\Delta_{\tau,\xi})|>1$ because $\gamma/2 \geq \gamma_2$.

Now, for $|\tau|\geqslant\gamma/2$, if $\eta\|\xi\|^2-\text{Re}(A)=0$, then $|\Delta_{\tau,\xi}| \geqslant |\text{Re} (\Delta_{\tau,\xi})|>1$ since $\gamma/2\geqslant\gamma_3$. In the case where $\eta\|\xi\|^2-\text{Re}(A)\neq 0$, we have $\gamma/2>0$, implying $|\tau|\geqslant 1$. Consequently, 
\begin{equation*}
	|\Delta_{\tau,\xi}|\geq|\text{Im}(\Delta_{\tau,\xi})|\geqslant 2|\eta^2\|\xi\|^2-\text{Re}(A)|\geqslant C>0,
\end{equation*}
where $C=\inf\{|\eta^2\|\xi\|^2-\text{Re}(A)|:\xi\in\mathbb{Z}^n \mbox{ and } \eta^2\|\xi\|^2-\text{Re}(A)\neq 0\}.$
Notice that $C>0$ because the set $\{\xi\in\mathbb{Z}^n:\eta^2\|\xi\|^2-\text{Re}(A)=0\}$ is closed.

\medskip
Then, if $\|\xi\|+|\tau|\geqslant\gamma$, $|\Delta_{\tau,\xi}|\geqslant C>0$, implying that \eqref{DC_condition} holds. This, in turn, leads to the conclusion from Theorem \ref{DC} that $P$ is solvable. The global hypoellipticity follows from Corollary \ref{GH}.
\end{proof}

The hypoellipticity of the heat operator is a well-established result in the literature. The preceding theorem establishes that the global hypoellipticity of the heat operator remains stable under both zero-order perturbations and zero-order conjugate perturbations.

\begin{theorem}[Wave operator]\label{wave}
	If $L={\partial^2}/{\partial t^2}-\eta^2\sum_{j=1}^n{\partial^2}/{\partial x_j^2}$, where $\eta>0$, then $P$ is solvable (and globally hypoelliptic) if and only if one of the following conditions holds:
	\begin{enumerate}
		\item[(i)] $|B|<|\text{Im}(A)|$;
		\item[(ii)] $|A|=|B|$, $\text{Re}(A)=0$ and $\eta$ is an irracional non-Liouville number;
		\item[(iii)] \eqref{DC_condition} holds.
	\end{enumerate}
\end{theorem}
\begin{proof}
Let us prove that the discriminant $\Delta_{\tau,\xi}$ satisfies the condition \eqref{DC_condition} of Theorem \ref{DC}. First, observe that the symbol 
$$
\sigma_L(\tau,\xi)=\overline{\sigma_L(-\xi,-\tau)}=-\tau^2+\eta^2\|\xi\|^2,  
$$
for all $(\tau,\xi)\in \mathbb{Z}^{n+1}.$

It follows from \eqref{Delta_xi} that
\begin{equation*}
	\Delta_{\tau,\xi} = \Big(\tau^2-\eta^2\|\xi\|^2+\text{Re}(A)\Big)^2+\text{Im}(A)^2-|B|^2, \ (\tau,\xi)\in \mathbb{Z}^{n+1}.
\end{equation*}

Assuming condition (i) holds, we have
\begin{equation*}
	|\Delta_{\tau,\xi}|=\Delta_{\tau,\xi}\geqslant \text{Im}(A)^2-|B|^2>0,  
\end{equation*} 
for all $(\tau,\xi)\in\mathbb{Z}^{n+1}.$ Therefore, (DC) holds and $P$ is solvable.

\medskip
Now, assume that condition (ii) holds, then
\begin{equation*}
	|\Delta_{\tau,\xi}|=|-\tau^2+\eta^2\|\xi\|^2|^2=|\tau-\eta\|\xi\||^2\cdot|\tau+\eta\|\xi\||^2.
\end{equation*}

Given that $\eta$ is an irrational non-Liouville number, there exist positive constants $C$ and $\gamma_0$ such that
\begin{equation*}
	\big|\tau\pm\eta\|\xi\|\big|\geqslant C(\|\xi\|+|\tau|)^{-\gamma_0},
\end{equation*}
for all $(\tau,\xi)\in(\mathbb{Z}^n\setminus\{0\})\times\mathbb{Z}.$ Consequently, we obtain
\begin{equation*}
	|\Delta_{\tau,\xi}|\geqslant C^4(\|\xi\|+|\tau|)^{-\gamma_0}=C^4(\|\xi\|+|\tau|)^{\gamma_0}(\|\xi\|+|\tau|)^{-5\gamma_0}.
\end{equation*}

Take $\gamma_1>0$ such that $\|\xi\|+|\tau|\geqslant\gamma_1$ implies $C^4(\|\xi\|+|\tau|)^{\gamma_0}\geqslant 1$, and set $\gamma=\max\{\gamma_1,5\gamma_0\}$. Consequently,
\begin{equation*}
	\|\xi\|+|\tau|\geqslant\gamma\ \implies |\Delta_{\tau,\xi}|\geqslant (\|\xi\|+|\tau|)^{-\gamma}.
\end{equation*}

Hence, \eqref{DC_condition} holds, and $P$ is solvable. 

Lastly, if condition (iii) holds, then $P$ is solvable directly.

On the other hand, if none of the conditions (i)-(iii) holds, then, in particular, \eqref{DC_condition} does not hold. Consequently, according to Theorem \ref{DC}, $P$ is not solvable. Finally, the global hypoellipticity prevails, as indicated by Corollary \ref{GH}.
\end{proof}

Finally, we recover  the following result concerning complex vector fields from \cite{BDM}.

\begin{theorem}[Complex vector fields]
The operator $P:\mathcal{C}^\infty(\mathbb{T}^2)\to \mathcal{C}^\infty(\mathbb{T}^2)$  given by
$$
Pu = \frac{\partial u}{\partial t} + C\frac{\partial u}{\partial x} - Au - \bar{B}u,
$$
with $A,B,C \in \mathbb{C}$, is solvable if and only if one of the following situations occurs:
\begin{enumerate}
		\item[(i)] $|B|>|A|$;
		\item[(ii)] $\text{Im}(C)\neq 0$;
		\item[(iii)] $|B|<|A|$ and $\text{Re}(A)\neq 0$;
		\item[(iv)] The pair $(C,\sqrt{|A|^2-|B|^2})$ belongs to $\mathbb{R}^2$ and satisfies \eqref{DC_condition}.
\end{enumerate}
\end{theorem}

\bibliographystyle{plain}
\bibliography{references}

\begin{thebibliography}{10}

\bibitem{BCP}
A.~P. Bergamasco, P.~D. Cordaro, and G.~Petronilho.
\newblock Global solvability for a class of complex vector fields on the
  two-torus.
\newblock {\em Commun. Partial Differ. Equations}, 29(5-6):785--819, 2004.

\bibitem{BDG2017_jfaa}
A.~P. Bergamasco, P.~L. Dattori~da Silva, and R.~B. Gonzalez.
\newblock Existence and regularity of periodic solutions to certain first-order
  partial differential equations.
\newblock {\em J. Fourier Anal. Appl.}, 23(1):65--90, 2017.

\bibitem{BDGK2015_jpdo}
A.~P. Bergamasco, P.~L. Dattori~da Silva, R.~B. Gonzalez, and A.~Kirilov.
\newblock Global solvability and global hypoellipticity for a class of complex
  vector fields on the 3-torus.
\newblock {\em J. Pseudo-Differ. Oper. Appl.}, 6(3):341--360, 2015.

\bibitem{BDM}
A.~P. Bergamasco, P.~L. Dattori~da Silva, and A.~Meziani.
\newblock Solvability of a first order differential operator on the two-torus.
\newblock {\em J. Math. Anal. Appl.}, 416(1):166--180, 2014.

\bibitem{AD}
M.~de~Almeida and P.~L. Dattori~da Silva.
\newblock Solvability of a class of first order differential operators on the
  torus.
\newblock {\em Result. Math.}, 76(2):17, 2021.
\newblock Id/No 104.

\bibitem{LeMe22_2}
B.~de~Lessa~Victor and A.~Meziani.
\newblock Infinitesimal bendings for classes of two-dimensional surfaces.
\newblock {\em Complex Variables and Elliptic Equations}, 0(0):1--23, 2022.

\bibitem{wagner}
W.~A.~A. de~Moraes.
\newblock Regularity of solutions to a {V}ekua-type equation on compact {L}ie
  groups.
\newblock {\em Ann. Mat. Pura Appl. (4)}, 201(1):379--401, 2021.

\bibitem{hounie}
J.~Hounie.
\newblock Globally hypoelliptic and globally solvable first-order evolution
  equations.
\newblock {\em Trans. Amer. Math. Soc.}, 252:233--248, 1979.

\bibitem{KMR2020_bsm}
A.~Kirilov, W.~A.~A. de~Moraes, and M.~Ruzhansky.
\newblock Partial {Fourier} series on compact {Lie} groups.
\newblock {\em Bull. Sci. Math.}, 160:27, 2020.
\newblock Id/No 102853.

\bibitem{KMR2021_jfa}
A.~Kirilov, W.~A.~A. de~Moraes, and M.~Ruzhansky.
\newblock Global hypoellipticity and global solvability for vector fields on
  compact {Lie} groups.
\newblock {\em J. Funct. Anal.}, 280(2):39, 2021.
\newblock Id/No 108806.

\bibitem{kravchenko}
V.~V. Kravchenko.
\newblock {\em Applied pseudoanalytic function theory}.
\newblock Front. Math. Basel: Birkh{\"a}user, 2009.

\bibitem{mez2}
A.~Meziani.
\newblock Solvability of planar complex vector fields with applications to
  deformation of surfaces.
\newblock In {\em Complex analysis}, Trends Math., pages 263--278.
  Birkh\"{a}user/Springer Basel AG, Basel, 2010.

\bibitem{mez4}
A.~Meziani.
\newblock Nonrigidity of a class of two dimensional surfaces with positive
  curvature and planar points.
\newblock {\em Proc. Amer. Math. Soc.}, 141(6):2137--2143, 2013.

\bibitem{Petr2011_tams}
G.~Petronilho.
\newblock Global hypoellipticity, global solvability and normal form for a
  class of real vector fields on a torus and application.
\newblock {\em Trans. Am. Math. Soc.}, 363(12):6337--6349, 2011.

\bibitem{vekua}
I.~N. Vekua.
\newblock Generalized analytic functions. {Transl}. ed. by {Ian} {N}.
  {Sneddon}.
\newblock International {Series} of {Monographs} on {Pure} and {Applied}
  {Mathematics}. {Vol}. 25. {Oxford}-{London}-{New} {York}-{Paris}: {Pergamon}
  {Press} 1962; {Reading}, {Mass}.- {London}: {Addison}-{Wesley} {Publ}. {Co}.,
  {Inc}. xxvi, 668 p. (1962)., 1962.

\end{thebibliography}

% ------------------------------------------------------------------------
\end{document}